\documentclass{amsart}
\usepackage{amsfonts}
\usepackage{amssymb}
\usepackage{times}
\usepackage{xcolor}
\usepackage{tkz-graph}
\usepackage{url}

\usepackage{hyperref}
\usepackage{tikz,pgf}
\usetikzlibrary{shapes}
\usetikzlibrary{calc,arrows.meta,positioning}

\def\B{\mathcal{B}}
\def\C{\mathcal{C}}

\def\I{\mathcal{I}}

\def\IF{\ensuremath{\mathbb F}}

\newcommand{\black}{\color{black}}

\newcounter{marg}[section]

\usepackage[normalem]{ulem}
\newcommand\rsout{\bgroup\markoverwith{\textcolor{red}{\rule[0.5ex]{2pt}{0.4pt}}}\ULon}

\theoremstyle{plain}%
  \newtheorem{theorem}{Theorem}[section]
  \newtheorem{corollary}{Corollary}[section]
  \newtheorem{proposition}{Proposition}[section]
  \newtheorem{lemma}{Lemma}[section]
  \newtheorem{example}{Example}[section]
  
  \newtheorem{definition}{Definition}[section]

\newtheorem{remark}{Remark}[section]

\newfont{\hueca}{msbm10}
\def\hu #1{\hbox{\hueca #1}}\def\hu #1{\hbox{\hueca #1}}

\begin{document}

\title{Leibniz algebras and graphs}

\thanks{The first author was  supported by the Centre for Mathematics of the University of Coimbra - UIDB/00324/2020, funded by the Portuguese Government through FCT/MCTES. Second author and fourth are supported by the PCI of the UCA `Teor\'\i a de Lie y Teor\'\i a de Espacios de Banach' and by the PAI with project number FQM298. Additionally, second author also is supported by the project FEDER-UCA18-107643. The third author was partially supported by CMUP, which is financed by national funds through FCT---Funda\c c\~ao para a Ci\^encia e a Tecnologia, I.P., under the project with reference UIDB/00144/2020.}

\author[E. Barreiro]{Elisabete Barreiro}
\address{Elisabete~Barreiro. \newline \indent University of Coimbra, CMUC, Department of Mathematics, Apartado 3008,
EC Santa Cruz,
3001-501 Coimbra
(Portugal). \hspace{0.1cm} {\em E-mail address}: {\tt mefb@mat.uc.pt}}{}

\author[A.J. Calder\'on]{Antonio J. Calder\'on}
\address{Antonio J. Calder\'on. \newline \indent Departamento de Matem\'aticas, Universidad de C\'adiz, Puerto Real (Espa\~na).}
\email{{\tt ajesus.calderon@uca.es}}

\author[S. A. Lopes]{Samuel A. Lopes}
\address{Samuel A. Lopes. \newline \indent CMUP, Departamento de Matem\'atica, Faculdade de Ci\^encias, Universidade do Porto, Rua do Campo Alegre s/n, 4169--007 Porto (Portugal).
\hspace{0.1cm} {\em E-mail address}: {\tt slopes@fc.up.pt}}{}

\author[J.M. S\'{a}nchez]{Jos\'{e} M. S\'{a}nchez}
\address{Jos\'{e} M. S\'{a}nchez. \newline \indent Departamento de Matem\'aticas, Universidad de C\'adiz, Puerto Real (Espa\~na).}
\email{{\tt txema.sanchez@uca.es}}


\thispagestyle{empty}

\begin{abstract}
We consider a Leibniz algebra ${\mathfrak L} = {\mathfrak I} \oplus {\mathfrak V}$ over an arbitrary base field $\mathbb{F}$, being ${\mathfrak I}$ the ideal generated by the products $[x,x], x \in {\mathfrak L}$. This ideal has a fundamental role in the study presented in our paper. A basis $\B=\{v_i\}_{i \in I}$ of ${\mathfrak L}$ is called multiplicative if for any $i,j \in I$ we have that $[v_i,v_j] \in {\mathbb F}v_k$ for some $k \in I$. We associate an adequate graph $\Gamma({\mathfrak L},\B)$ to ${\mathfrak L}$ relative to $\B$.  By arguing on this graph we show that ${\mathfrak L}$ decomposes as a direct sum of ideals, each one being associated to one connected component of $\Gamma({\mathfrak L},\B)$. Also the minimality of ${\mathfrak L}$ and the division property of ${\mathfrak L}$ are characterized in terms of the weak symmetry of the defined subgraphs $\Gamma({\mathfrak L},\B_{\mathfrak I})$ and $\Gamma({\mathfrak L},\B_{\mathfrak V})$.

\bigskip

{\it 2020MSC}: 05C50, 05C25, 17A32, 17A60, 17B65.

{\it Keywords}: Leibniz algebra, multiplicative basis, graph, structure theory.  \black
\end{abstract}

\maketitle

\section{Introduction}\label{sec1}

Leibniz algebras were presented by Bloh \cite{Leibniz3}, who called them the $D$-algebras. Two decades later, Loday introduced them in \cite{Loday} as a non-antisymmetric analogue of Lie algebras and naming them Leibniz algebras because it was Gottfried W. Leibniz who discovered the {\em Leibniz rule} for differentiation of functions. This formula is realized when the adjoint map is a derivation and generalizes the well-known Jacobi identity in such a way that any Lie algebra is a Leibniz algebra.

Recently, the structure of Leibniz algebras has been considered in the frameworks of low dimensional algebras, nilpotency and related problems \cite{Alb, Ayu99, Ayu01, Cabezas, Casas, Jiang, Liu}. We finally have to mention the recent book \cite{new_book}, where a complete state of the art on Leibniz algebras theory can be found. The inner structure of Leibniz algebras admitting a multiplicative basis $\B$ has been recently studied in \cite{Yo_Leibniz_mult_basis}, where the focus is on the  characterization of the $\B$-semisimplicity and of the $\B$-simplicity  of the algebra.

An interesting problem in graph theory and in abstract algebra consists in characterizing the structure of an algebraic object by the properties satisfied for some graph associated with it (see for instance \cite{big1,big2,graph3,graph4}). The paper \cite{Yografos} is devoted to the study of the structure of linear spaces under a bilinear map, by associating to this data an adequate graph.

The main goal of the present paper is to use properties of graphs to study Leibniz algebras ${\mathfrak L}$ with multiplicative bases $\B$, in order to obtain results about their algebraic structure.
Given a Leibniz algebras ${\mathfrak L}$, the ideal ${\mathfrak I}$ generated by the products $[x,x]$ with $x \in {\mathfrak L}$ plays a relevant role in our study, so we have to adapt the ideas of \cite{Yografos} to this situation.
By arguing on the associated graph, we show that a Leibniz algebra decomposes as a direct sum of ideals, each one being associated to one connected component of the mentioned graph.
In the last section, we approach the minimality property for this class of algebras by using the subgraphs $\Gamma({\mathfrak L},\B_{\mathfrak I})$ and $\Gamma({\mathfrak L},\B_{\mathfrak V})$, related to a given (vector space) decomposition ${\mathfrak L}={\mathfrak I}\oplus {\mathfrak V}$.


The paper is organized as follows. In Section 2 we introduce the (directed) graph associated with a Leibniz algebra ${\mathfrak L}$ admitting a multiplicative basis $\B$, denoted as $\Gamma({\mathfrak L},\B)$. By using this graph we prove that ${\mathfrak L}$ decomposes as a direct sum $${\mathfrak L} = \bigoplus_k \I_k$$ of ideals with a multiplicative basis contained in $\B$, each one being associated with a connected component of $\Gamma({\mathfrak L},\B)$. In the next section we discuss the relations among various decompositions of ${\mathfrak L}$ given by different choices of bases. Finally, in Section 4 we relate the weak symmetry of two concrete subgraphs with some properties of ${\mathfrak L}$.  
The minimality of ${\mathfrak L}$ and the division property of $\B$ are characterized. It is shown that ${\mathfrak L}$ is minimal if and only if $\B$ is of division if and only if the two graphs $\Gamma({\mathfrak L},\B_{\mathfrak I})$ and $\Gamma({\mathfrak L},\B_{\mathfrak V})$ are weakly symmetric. 

All of the Leibniz algebras considered are of arbitrary (finite) dimension over an arbitrary base field $\mathbb{F}$.

\begin{definition}\rm
A {\em Leibniz algebra}  ${\mathfrak L}$ is a vector space over a
field $\hu{F}$ endowed with a bilinear product $[\cdot, \cdot]$ satisfying the {\it Leibniz identity}
$$ [[y,z],x]=[[y,x],z]+[y,[z,x]],$$ for any $x, y, z \in {\mathfrak L}$.
\end{definition}

\noindent Clearly, Lie algebras are examples of Leibniz algebras. For any $x \in {\mathfrak L}$, consider the   adjoint map  ${\rm ad}_x : {\mathfrak L} \to {\mathfrak L}$ defined by ${\rm ad}_x(y) = [y,x]$, with $y \in {\mathfrak L}$. Observe that the Leibniz identity is equivalent to the assertion that ${\rm ad}_x$ is a derivation of ${\mathfrak L}$, for all $x \in {\mathfrak L}$.
A {\it subalgebra} $S$ of a Leibniz algebra ${\mathfrak L}$ is a vector subspace of ${\mathfrak L}$ such that $[S,S]\subset S$, and an {\it ideal} $I$ of ${\mathfrak L}$ is a subalgebra such that $[I,{\mathfrak L}]+[{\mathfrak L},I]\subset I$. 

The ideal ${\mathfrak I}$ generated by $\{[x,x]: x\in {\mathfrak L}\}$ plays an important role in the theory since it determines the (possible)  non-Lie character of the Leibniz algebra ${\mathfrak L}$. From the Leibniz identity, this ideal satisfies
\begin{equation}\label{equi}
[{\mathfrak L},{\mathfrak I}]=0.
\end{equation}

\noindent Observe that we can write $${\mathfrak L}= {\mathfrak I} \oplus {\mathfrak V},$$ where ${\mathfrak V}$ is a linear complement of ${\mathfrak I}$ in ${\mathfrak L}$ (as a vector space, ${\mathfrak V}$ is isomorphic to ${\mathfrak L}/{\mathfrak I}$, the so-called corresponding Lie algebra of  ${\mathfrak L}$). Hence, by taking bases ${\mathcal B}_{\mathfrak I}$ and $ {\mathcal B}_{\mathfrak V}$ of ${\mathfrak I}$ and $ {\mathfrak V}$, respectively, we get $${\mathcal B}={\mathcal B}_{\mathfrak I} \; \dot{\cup} \; {\mathcal B}_{\mathfrak V},$$ a basis of ${\mathfrak L}$.

\begin{definition}\label{def_orthogonal}\rm
A decomposition of a Leibniz algebra ${\mathfrak L}$ as the direct sum of linear subspaces $${\mathfrak L} = \bigoplus_{j\in J} {\mathfrak L}_j$$ is called \emph{orthogonal} if $[{\mathfrak L}_j,{\mathfrak L}_k]=\{0\}$ for all $j \neq k \in J$.
\end{definition}

For an arbitrary algebra $A$ over a base field $\mathbb{F}$,  a basis ${\mathcal B}=\{v_i\}_{i \in I}$ of $A$ is called {\em multiplicative} if for any $i,j \in I$ we have that  $v_iv_j \in {\mathbb F}v_k$, for some $k \in I$ (see \cite{mbasis}). In the particular case of Leibniz algebras, as in \cite{Yo_Leibniz_mult_basis}, the relation \eqref{equi} implies the following characterization.

\begin{definition}\label{def_mult_Leib}\rm
A basis ${\mathcal B}={\mathcal B}_{\mathfrak I} \; \dot{\cup} \; {\mathcal B}_{\mathfrak V}$ of a Leibniz algebra ${\mathfrak L},$ where we denote $${\mathcal B}_{\mathfrak I}:=\{e_k\}_{k \in K}, \hspace{0.4cm} {\mathcal B}_{\mathfrak V} := \{u_j\}_{j \in J},$$ is called {\em multiplicative} if:
\begin{enumerate}
\item[i.] For any $k \in K$ and $j \in J$ we have $[e_k, u_j] \in \mathbb{F}e_i$ for some $i \in K$.
\item[ii.] For any $j,k \in J$ we have either $[u_j, u_k] \in \mathbb{F}u_l$ for some $l \in J$, or $[u_j, u_k] \in \mathbb{F}e_i$ for some $i \in K$.
\end{enumerate}
\end{definition}

\begin{example}  \label{example}\rm (see \cite{Omirov})
Consider a (non-Lie) Leibniz algebra ${\mathfrak L}$ over a field $\mathbb{F}$ with characteristic different of $2$ and with the multiplicative  basis $\B=\B_{\mathfrak I}\;\dot{\cup}\;B_{\mathfrak V},$ where $\B_{\mathfrak I}=\{p,q\}$ and $\B_{\mathfrak V}=\{e,h,f\}$, defined by the following multiplication:
$$[e,h]= -[h,e] = 2e, \hspace{0.2cm}[h,f] = -[f,h] = 2f, \hspace{0.2cm} [e,f] = -[f,e] = h,$$
$$[p,h] = -[q,e] = p, \hspace{0.2cm}[p,f] = -[q,h] = q, $$
where the omitted products among basis elements are zero.
\end{example}

\begin{example}\label{r3}\rm 
The infinite-dimensional Leibniz algebra ${\mathfrak L}={\mathfrak I} \oplus {\mathfrak V}$ given in \cite[Example 1.2]{Yo_Leibniz_mult_basis}, over a base field $\IF$ with characteristic different to $2$, admits a multiplicative basis ${\mathcal B}={\mathcal B}_{\mathfrak I} \;\dot{\cup} \; {\mathcal B}_{\mathfrak V}$, where $${\mathcal B}_{\mathfrak I}= \{e_n: n \in {\mathbb N} \}$$ is a basis of ${\mathfrak I}$ and $${\mathcal B}_{\mathfrak V}=\{u_a,u_b,u_c, u_d\}$$ is a basis of ${\mathfrak V}$. The non-zero products from the basis ${\mathcal B} $ are:
\begin{center}
\begin{tabular}{l}
$[u_b,u_c]=u_a, \hspace{0.3cm} [u_c,u_b]=-u_a,$ \\
$[u_d,u_d]=e_0, \hspace{0.3cm} [e_0,u_d]=e_1,$\\
$[e_n,u_a]=e_n, \mbox{ for } n\geq 2,$\\
$[e_n,u_b]=e_{n+1}, \mbox{ for } n\geq 2,$\\
$[e_n,u_c]=(n-2)e_{n-1}, \mbox{ for } n\geq 3.$
\end{tabular}
\end{center}
\end{example}

\begin{example}\label{exemplo_filiform}
Let ${\mathfrak L}$ be the model filiform Leibniz algebra with multiplicative basis $\mathcal{B} =  \mathcal{B}_{\mathfrak I} \;\dot{\cup} \;    \mathcal{B}_{\mathfrak V}$, where $  \mathcal{B}_{\mathfrak I}  =  \{e_2,\dots,e_n\}$  and $  \mathcal{B}_{\mathfrak V} =  \{e_1 \}$, with non-zero products given by $$[e_i, e_1] = e_{i+1}, \hspace{0.2cm} 2 \leq i \leq n - 1.$$ 
\end{example}

\begin{example}\label{Example_Omirov}
Let $n \geq 4$ be an odd integer and ${\mathfrak L}$ be the  
$n$-dimensional  Leibniz algebra, given in Theorem 3.4 of \cite{ejemploOmirov} over the complex numbers, with  basis $$\mathcal{B} := \{e,f,h,x_0,x_1,\dots, x_{n-4}\},$$ such that the non-zero products from $\mathcal{B}$ are
\begin{align*}
& [e,f] = -[f,e] = h, && [x_k, e] = k(k+3-n)x_{k-1}, \hspace{0.2cm} 1 \leq k \leq n-4, \\
& [e,h] = -[h,e] = 2e, && [x_k, h] = (n-4-2k)x_k, \hspace{0.2cm} 0 \leq k \leq n-4,\\
& [h,f] = -[f,h] = 2 f, && [x_k, f] = x_{k+1}, \hspace{0.2cm}  0 \leq k \leq n-5.
\end{align*}
Note that ${\mathfrak L}/{\mathfrak I}$ is isomorphic to the simple Lie algebra $sl_2$.
Clearly $\mathcal{B}$ is a multiplicative basis. This example generalizes Example~\ref{example}, with $n=5$ and $p=x_0$, $q=x_1$.
\end{example}


\section{Leibniz algebras admitting a multiplicative basis and their graphs}

In this section we relate the concept of Leibniz algebra admitting a multiplicative basis with graphs in order to characterize their structure. We recall that a \emph{(directed) graph} is a pair $(V,E)$ where $V$ is a set of vertices and $E\subset V\times V$ a set of (directed) edges connecting the vertices.
The reader can find the basic necessary graph concepts in \cite{Yografos}. All graphs considered here will be directed graphs.

\begin{definition}\label{def_graph_associated}\rm
Let ${\mathfrak L}$ be a Leibniz algebra admitting a multiplicative basis $\B$. The  directed graph associated to ${\mathfrak L}$ relative to $\B$ is $\Gamma({\mathfrak L},\B) := (V,E)$, where $V:=\B$ and
$$E := \{(v_j,v_k) \in V \times V : \{[v_i,v_j],[v_j,v_i]\} \cap \mathbb{F}^{\times}v_k \neq \emptyset, \hspace{0.1cm} \mbox{\rm for some } v_i \in \B\}.$$
\end{definition}

\begin{remark}\label{Remark21} \rm
Taking into account Definition \ref{def_mult_Leib}, we observe that $\Gamma({\mathfrak L},\B)$ can be described more precisely by $V=\B_{\mathfrak I}\; \dot{\cup} \; \B_{\mathfrak V}$ and
\begin{align*}
E &= \{(u_j,e_k) \in V \times V : [e_i,u_j]\in \mathbb{F}^{\times}e_k \neq \emptyset, \hspace{0.1cm} \mbox{\rm for some } e_i \in \B_{\mathfrak I}\} \\
& \; \cup \{(e_i,e_k) \in V \times V : [e_i,u_j]\in \mathbb{F}^{\times}e_k \neq \emptyset, \hspace{0.1cm} \mbox{\rm for some } u_j \in \B_{\mathfrak V}\} \\
& \; \cup \{(u_j,e_i) \in V \times V : \{[u_j,u_k],[u_k,u_j]\} \cap \mathbb{F}^{\times}e_i \neq \emptyset, \hspace{0.1cm} \mbox{\rm for some } u_k \in \B_{\mathfrak V}\} \\
& \; \cup \{(u_j,u_l) \in V \times V : \{[u_j,u_k],[u_k,u_j]\} \cap \mathbb{F}^{\times}u_l \neq \emptyset, \hspace{0.1cm} \mbox{\rm for some } u_k \in \B_{\mathfrak V}\}.
\end{align*}
\end{remark}




\begin{example}\label{Ejemplo_grafo}\rm
Consider the Leibniz algebra ${\mathfrak L}$ of the Example \ref{example} over a field with characteristic different of $2$ and with the basis $\B=\B_{\mathfrak I}\;\dot{\cup}\;B_{\mathfrak V},$ where $\B_{\mathfrak I}=\{p,q\}$, $\B_{\mathfrak V}=\{e,h,f\}$.  
We have that $V=\{e,h,f,p,q\}$,
\begin{align*}
E = &\{(e,p),(h,p),(h,q),(f,q)\} \; \cup \; \{(p,p),(p,q),(q,p),(q,q)\} \\
& \; \cup \; \emptyset \; \cup \; \{(e,e),(e,h),(h,f),(f,f),(h,e),(f,h)\}
\end{align*}
and the associated graph $\Gamma({\mathfrak L},\B)$  is the following:
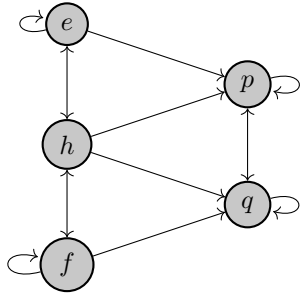
\begin{center}
\begin{tikzpicture}[scale=0.8, transparency group=knockout]
\begin{scope}[every node/.style={circle,thick,draw}]
\node[shape=circle,fill={rgb,255:gray,300; white,50},draw=black] (E) at (-1,0) {$e$};
\node[shape=circle,fill={rgb,255:gray,300; white,50},draw=black] (P) at (2,-1) {$p$};
\node[shape=circle,fill={rgb,255:gray,300; white,50},draw=black] (H) at (-1,-2) {$h$};
\node[shape=circle,fill={rgb,255:gray,300; white,50},draw=black] (F) at (-1,-4) {$f$};
\node[shape=circle,fill={rgb,255:gray,300; white,50},draw=black] (Q) at (2,-3) {$q$};

\draw[<->] (E) edge (H);
\draw[->] (E) edge (P);
\draw[->] (H) edge (P);
\draw[->] (H) edge (Q);
\draw[->] (F) edge (Q);
\draw[<->] (H) edge (F);
\draw[<->] (P) edge (Q);

\path (E) edge [loop left] (E);
\path (P) edge [loop right] (P);
\path (Q) edge [loop right] (Q);
\path (F) edge [loop left] (F);
\end{scope}
\end{tikzpicture}
\end{center}
\end{example}

\begin{example}\label{r3_grafo}\rm 
The associated graph of the Leibniz algebra presented in Example \ref{r3} over a field with characteristic $0$ and with the  multiplicative basis ${\mathcal B}={\mathcal B}_{\mathfrak I} \;\dot{\cup} \; {\mathcal B}_{\mathfrak V}$, where ${\mathcal B}_{\mathfrak I}= \{e_n: n \in {\mathbb N} \}$ and ${\mathcal B}_{\mathfrak V}=\{u_a,u_b,u_c, u_d\}$ is as follows:
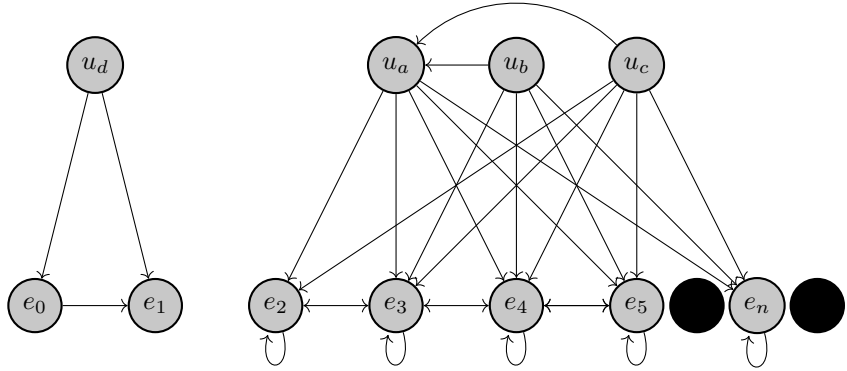
\begin{center}
\begin{tikzpicture}[scale=0.8, transparency group=knockout]
\begin{scope}[every node/.style={circle,thick,draw}]
\node[shape=circle,fill={rgb,255:gray,300; white,50},draw=black] (D) at (-1,2) {$u_d$};
\node[shape=circle,fill={rgb,255:gray,300; white,50},draw=black] (A) at (4,2) {$u_a$};
\node[shape=circle,fill={rgb,255:gray,300; white,50},draw=black] (B) at (6,2) {$u_b$};
\node[shape=circle,fill={rgb,255:gray,300; white,50},draw=black] (C) at (8,2) {$u_c$};
\node[shape=circle,fill={rgb,255:gray,300; white,50},draw=black] (E0) at (-2,-2) {$e_0$};
\node[shape=circle,fill={rgb,255:gray,300; white,50},draw=black] (E1) at (0,-2) {$e_1$};
\node[shape=circle,fill={rgb,255:gray,300; white,50},draw=black] (E2) at (2,-2) {$e_2$};
\node[shape=circle,fill={rgb,255:gray,300; white,50},draw=black] (E3) at (4,-2) {$e_3$} ;
\node[shape=circle,fill={rgb,255:gray,300; white,50},draw=black] (E4) at (6,-2) {$e_4$};
\node[shape=circle,fill={rgb,255:gray,300; white,50},draw=black] (E5) at (8,-2) {$e_5$};
\node[shape=circle,fill={rgb,255:gray,300; white,50},draw=black] (E) at (10,-2) {$e_n$};

\draw[->] (B) edge (A);
\draw[->] (C) edge [in=45, out=135] (A);

\draw[->] (E0) edge (E1);
\draw[<->] (E2) edge (E3);
\draw[<->] (E3) edge (E4);
\draw[<->] (E4) edge (E5);
\draw[<->] (E4) edge (E5);

\node[fill,opacity=0,text opacity=1] at ($(E5)!.5!(E)$) {\ldots};
\node[fill,opacity=0,text opacity=1] at ($(E)!.5!(12,-2)$) {\ldots};

\draw[->] (D) edge (E0);
\draw[->] (D) edge (E1);
\draw[->] (A) edge (E2);
\draw[->] (A) edge (E3);
\draw[->] (A) edge (E4);
\draw[->] (A) edge (E5);
\draw[->] (A) edge (E);
\draw[->] (B) edge (E3);
\draw[->] (B) edge (E4);
\draw[->] (B) edge (E5);
\draw[->] (B) edge (E);
\draw[->] (C) edge (E2);
\draw[->] (C) edge (E3);
\draw[->] (C) edge (E4);
\draw[->] (C) edge (E5);
\draw[->] (C) edge (E);

\path (E2) edge [loop below] (E2);
\path (E3) edge [loop below] (E3);
\path (E4) edge [loop below] (E4);
\path (E5) edge [loop below] (E5);
\path (E) edge [loop below] (E);
\end{scope}
\end{tikzpicture}
\end{center}
\end{example}

\begin{example}
For the model filiform Leibniz algebra ${\mathfrak L}$ with basis $\mathcal{B} := \{e_1,\dots,e_n\}$ from Example \ref{exemplo_filiform} the associated graph $\Gamma({\mathfrak L},\mathcal{B})$ is
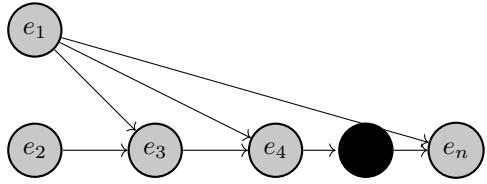
\begin{center}
\begin{tikzpicture}[scale=0.8, transparency group=knockout]
\begin{scope}[every node/.style={circle,thick,draw}]
\node[shape=circle,fill={rgb,255:gray,300; white,50},draw=black] (E1) at (4,2) {$e_1$};
\node[shape=circle,fill={rgb,255:gray,300; white,50},draw=black] (E2) at (4,0) {$e_2$} ;
\node[shape=circle,fill={rgb,255:gray,300; white,50},draw=black] (E3) at (6,0) {$e_3$};
\node[shape=circle,fill={rgb,255:gray,300; white,50},draw=black] (E4) at (8,0) {$e_4$};
\node[shape=circle,fill={rgb,255:gray,300; white,50},draw=black] (En) at (11,0) {$e_n$};

\draw[->] (E1) edge (E3);
\draw[->] (E1) edge (E4);
\draw[->] (E1) edge (En);
\draw[->] (E2) edge (E3);
\draw[->] (E3) edge (E4);
\draw[->] (E4) edge (9,0);
\draw[<-] (En) edge (9.9,0);

\node[fill,opacity=0,text opacity=1] at ($(E4)!.5!(En)$) {\ldots};
\end{scope}
\end{tikzpicture}
\end{center}
\end{example}

\begin{example}\label{EJEMM_24}
Let ${\mathfrak L}$ be the $n$-dimensional  Leibniz algebra from Example \ref{Example_Omirov} over a field with characteristic $0$. The associated graph $\Gamma({\mathfrak L},\mathcal{B})$ is 
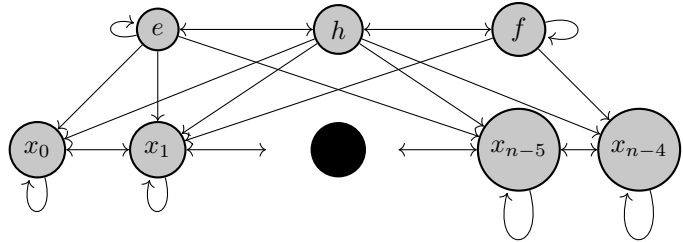
\begin{center}
\begin{tikzpicture}[scale=0.8, transparency group=knockout]
\begin{scope}[every node/.style={circle,thick,draw}]
\node[shape=circle,fill={rgb,255:gray,300; white,50},draw=black] (E) at (-1,2) {$e$};
\node[shape=circle,fill={rgb,255:gray,300; white,50},draw=black] (H) at (2,2) {$h$};
\node[shape=circle,fill={rgb,255:gray,300; white,50},draw=black] (F) at (5,2) {$f$};
\node[shape=circle,fill={rgb,255:gray,300; white,50},draw=black] (X0) at (-3,0) {$x_0$};
\node[shape=circle,fill={rgb,255:gray,300; white,50},draw=black] (X1) at (-1,0) {$x_1$};
\node[shape=circle,fill={rgb,255:gray,300; white,50},draw=black] (XN-5) at (5,0) {$x_{n-5}$};
\node[shape=circle,fill={rgb,255:gray,300; white,50},draw=black] (XN-4) at (7,0) {$x_{n-4}$};

\draw[->] (E) edge (X0);
\draw[->] (E) edge (X1);
\draw[->] (E) edge (XN-5);
\draw[<->] (E) edge (H);
\draw[<->] (H) edge (F);
\draw[->] (H) edge (X0);
\draw[->] (H) edge (X1);
\draw[->] (H) edge (XN-5);
\draw[->] (H) edge (XN-4);
\draw[->] (F) edge (X1);
\draw[->] (F) edge (XN-4);
\draw[<->] (X0) edge (X1);
\draw[<->] (X1) edge (0.8,0);
\draw[<->] (XN-5) edge (3,0);
\draw[<->] (XN-5) edge (XN-4);


\node[fill,opacity=0,text opacity=1] at (2,0) {\ldots};

\path (E) edge [loop left] (E);
\path (F) edge [loop right] (F);
\path (X0) edge [loop below] (X0);
\path (X1) edge [loop below] (X1);
\path (XN-5) edge [loop below] (XN-5);
\path (XN-4) edge [loop below] (XN-4);
\end{scope}
\end{tikzpicture}
\end{center}
\end{example}

Given two vertices $v_i,v_j\in V$, an \emph{undirected path} from $v_i$ to $v_j$ is a sequence of vertices $(v_{i_1},\dots,v_{i_n})$ with $v_{i_1}=v_i$,  $v_{i_n}=v_j$ and such that either $(v_{i_{r}},v_{i_{r+1}})\in E$ or $(v_{i_{r+1}},v_{i_{r}})\in E$, for $1\le r \le n-1$.

We can introduce an  equivalence relation in $V$ defined by $v_i\sim v_j$ if and only if either $v_i= v_j$ or there exists  an undirected  path from $v_i$ to $v_j$. Then it is  said that $v_i$ and $v_j$ are \emph{connected} and the equivalence class of $v_i$, denoted by $[v_i] \in V/\sim$, corresponds to a connected component $\C_{[v_i]}$ of the graph $\Gamma({\mathfrak L},\B)$. Then
\begin{equation}\label{graphdec}
\Gamma({\mathfrak L},\B)=\dot{\bigcup_{[v_i]\in V/\sim}} \C_{[v_i]}.
\end{equation}

\noindent We can also associate to any $\C_{[v_i]}$ the linear subspace
\begin{equation} \label{linsub}
{\mathfrak L}_{\C_{[v_i]}}:=\bigoplus_{v_j\in [v_i]}\IF v_j
\end{equation}

\noindent and assert the next result:

\begin{theorem}\label{decompositions}
Let ${\mathfrak L}$ be a Leibniz algebra admitting a multiplicative basis $\B=\{v_i\}_{i\in I}$. Then ${\mathfrak L}$ decomposes as  the orthogonal direct sum $${\mathfrak L} = \bigoplus\limits_{[v_i] \in V/\sim} {\mathfrak L}_{C_{[v_i]}},$$ where each ${\mathfrak L}_{C_{[v_i]}}$ is an ideal of ${\mathfrak L}$, admitting the set $[v_i] \subset \B$ as multiplicative basis.
\end{theorem}
\begin{proof}
From Equation \eqref{graphdec} and Equation \eqref{linsub} we can assert that ${\mathfrak L}$ is the direct sum of the family of linear subspaces ${\mathfrak L}_{C_{[v_i]}}$ with $[v_i] \in V/\sim$, each one admitting the set $[v_i] \subset \B$ as multiplicative basis.

Let us suppose that there exist $v_j \in {\mathfrak L}_{C_{[v_i]}}$ and $v_k, v_l \in \B$ such that $$\{[v_j,v_k],[v_k,v_j]\} \cap \mathbb{F}^{\times}v_l \neq \emptyset.$$  Then $(v_j,v_l)$ and $(v_k,v_l)$ are edges of $C_{[v_i]}$, and thus $v_k,v_l \in {\mathfrak L}_{C_{[v_i]}}$. From here we conclude that the direct sum is orthogonal and that ${\mathfrak L}_{C_{[v_i]}}$ is an ideal of ${\mathfrak L}$.
\end{proof}


\begin{example} \rm 
The Leibniz algebra of Example \ref{r3} decomposes as the orthogonal direct sum $${\mathfrak L} = {\mathfrak L}_{C_{[u_d]}} \oplus {\mathfrak L}_{C_{[u_a]}}$$ where $${\mathfrak L}_{C_{[u_d]}} := \mathbb{F} u_d \oplus \mathbb{F}e_0 \oplus \mathbb{F}e_1, \hspace{1cm} {\mathfrak L}_{C_{[u_a]}} := \mathbb{F}u_a \oplus \mathbb{F}u_b \oplus \mathbb{F}u_c \oplus \mathbb{F}e_2 \oplus \mathbb{F}e_3 \oplus \cdots \oplus \mathbb{F}e_n \oplus \cdots$$ are ideals of ${\mathfrak L}$ admitting the multiplicative bases in $\B$: $$[u_d]:=\{u_d,e_0,e_1\}, \hspace{1.5cm} [u_a]:=\{u_a,u_b,u_c\}\cup \{e_i\}_{i\geq 2},$$ respectively.
\end{example}

We recall a Leibniz algebra ${\mathfrak L}$ is {\em simple} if its product is non-zero and its only ideals are $\{0\}$, ${\mathfrak I}$ and ${\mathfrak L}$. It should be noted that this definition agrees with the definition of simple Lie algebra, since in this framework ${\mathfrak I}=\{0\}$.

\begin{corollary}\label{Corolario 2.1}
If ${\mathfrak L}$ is simple, then any two vertices of $\Gamma({\mathfrak L},\B)$ are connected.
\end{corollary}

\begin{example}
By Corollary \ref{Corolario 2.1}, the Leibniz algebra of Example \ref{r3} is not simple because its associated graph has two components, as shown in Example \ref{r3_grafo}.
\end{example}

\noindent To identify the components of the decomposition given in Theorem \ref{decompositions} we only need to focus on the connected components of the associated graph.

\section{Relating the graphs given by different choices of bases}

In general, two different multiplicative bases of ${\mathfrak L}$ determine different associated graphs, which can give rise to different decompositions of ${\mathfrak L}$ as an orthogonal direct sum of ideals (see Theorem \ref{decompositions}). For instance, consider the next example. First we recall that two graphs $(V,E)$ and $(V',E')$ are {\em isomorphic} if there exists a bijection $f : V \to V'$ such that $(v_i,v_j) \in E$ if and only if $(f(v_i),f(v_j)) \in E'$.
 
\begin{example}
Consider the complex Leibniz algebra ${\mathfrak L}$ of Example \ref{example} with the multiplicative basis $\B':=\{e+f,e-f,h,p+q,p-q\}$. Indeed, the non-zero products in the basis $\B' $ are
\begin{align*}
&[e-f,e+f] = -[e+f,e-f] = 2h,\\
&[e+f,h] = -[h,e+f] = 2(e-f),\\
&[e-f,h] = -[h,e-f] = 2(e+f),\\
&[p+q,h] = -[p+q,e+f] = [p-q,e-f] = p-q,\\
&[p-q,e+f] = [p-q,h] = -[p+q,e-f] = p+q.
\end{align*}
Then the associated graph $\Gamma({\mathfrak L},\B')$ is
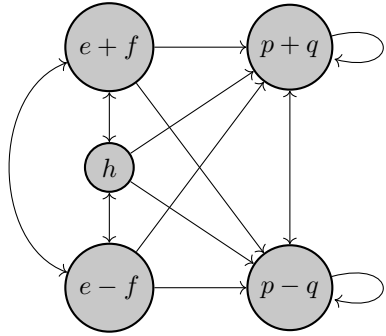
\begin{center}
\begin{tikzpicture}[scale=0.8, transparency group=knockout]
\begin{scope}[every node/.style={circle,thick,draw}]
\node[shape=circle,fill={rgb,255:gray,300; white,50},draw=black] (E+F) at (-1,0) {$e+f$};
\node[shape=circle,fill={rgb,255:gray,300; white,50},draw=black] (E-F) at (-1,-4) {$e-f$};
\node[shape=circle,fill={rgb,255:gray,300; white,50},draw=black] (P+Q) at (2,0) {$p+q$};
\node[shape=circle,fill={rgb,255:gray,300; white,50},draw=black] (P-Q) at (2,-4) {$p-q$};
\node[shape=circle,fill={rgb,255:gray,300; white,50},draw=black] (H) at (-1,-2) {$h$};

\draw[<->] (E+F) edge [in=-560, out=560] (E-F);
\draw[<->] (P+Q) edge (P-Q);
\draw[<->] (E+F) edge (H);
\draw[<->] (E-F) edge (H);
\draw[->] (E+F) edge (P+Q);
\draw[->] (E+F) edge (P-Q);
\draw[->] (E-F) edge (P+Q);
\draw[->] (E-F) edge (P-Q);
\draw[->] (H) edge (P+Q);
\draw[->] (H) edge (P-Q);

\path (P+Q) edge [loop right] (P+Q);
\path (P-Q) edge [loop right] (P-Q);
\end{scope}
\end{tikzpicture}
\end{center}
which is  clearly non-isomorphic to the one obtained with the multiplicative basis $\B$ presented in Example \ref{Ejemplo_grafo}.
\end{example}

Next we give a condition under which the graphs associated to ${\mathfrak L}$ by two different multiplicative bases are isomorphic. As a consequence, we establish a sufficient condition under which two decompositions of ${\mathfrak L}$, induced by two different multiplicative bases, are equivalent.

\noindent We recall that an {\em automorphism} of a Leibniz algebra ${\mathfrak L}$ is a linear isomorphism $f : {\mathfrak L} \to {\mathfrak L}$ satisfying $[f(x),f(y)]=f([x,y])$, for all $x,y \in {\mathfrak L}$.

\begin{definition}\rm
Let ${\mathfrak L}$ be a Leibniz algebra. Two bases $\B=\{v_i\}_{i \in I}$ and $\B'=\{w_j\}_{j \in J}$ of ${\mathfrak L}$ are {\em equivalent} if there exists an automorphism $f : {\mathfrak L} \to {\mathfrak L}$ satisfying $f(\B)=\B'$.
\end{definition}

\begin{lemma}\label{isomorphism}
Let ${\mathfrak L}$ be a Leibniz algebra and consider two multiplicative bases $\B$ and $\B'$ of ${\mathfrak L}$. If $\B$ and $\B'$ are two equivalent bases, then the associated graphs $\Gamma({\mathfrak L},\B)$ and $\Gamma({\mathfrak L},\B')$ are isomorphic.
\end{lemma}

\begin{proof}
Let us suppose that $\B=\{v_i\}_{i\in I}$ and $\B'=\{w_j\}_{j\in J}$ are two equivalent multiplicative bases of ${\mathfrak L}$. Then, there exists an automorphism $f : {\mathfrak L} \to {\mathfrak L}$ satisfying
\begin{equation}\label{eqq1}
[f(x),f(y)] = f([x,y])
\end{equation}
for all $x,y\in {\mathfrak L}$ and such that $f(\B)=\B'$.

Let us denote by $(V,E)$ and $(V',E')$ the set of vertices and edges of $\Gamma({\mathfrak L},\B)$ and $\Gamma({\mathfrak L},\B')$, respectively. Taking into account that $V=\B$ and $V'=\B'$, and the fact $f(\B)=\B'$, we have that the map $f$ defines a bijection from $V$ to $V'$.

Given $x,y \in V$, we want to show that $(x,y) \in E$ if and only if $(f(x),f(y)) \in E'$.  Suppose that $(x,y)\in E$. Then there are $z\in V$ and $\lambda\in\IF^\times$ such that either $[x, z]=\lambda y$ or $[z,x]=\lambda y$. Applying $f$ to these relations we get that either $[f(x), f(z)]=\lambda f(y)$ or $[f(z),f(x)]=\lambda f(y)$. Thus, $(f(x),f(y)) \in E'$. The same argument using $f^{-1}$ shows that if $(f(x),f(y)) \in E'$ then $(x,y) \in E$, which concludes the proof that $\Gamma({\mathfrak L},\B)$ and $\Gamma({\mathfrak L},\B')$ are isomorphic via $f$.

\end{proof}

The following concept is borrowed from the theory of graded algebras (see for instance \cite{f4}).

\begin{definition}\rm
Let ${\mathfrak L}$ be a Leibniz algebra and let $$\Upsilon := {\mathfrak L} = \bigoplus_{i \in I} {\mathfrak L}_i \hspace{0.4cm} \mbox{\rm and} \hspace{0.4cm} \Upsilon' := {\mathfrak L} = \bigoplus_{j \in J} {\mathfrak L}'_j$$ be two decompositions of ${\mathfrak L}$ as an orthogonal direct sum of ideals. It is said that $\Upsilon$ and $\Upsilon'$ are {\em equivalent} if there exists an automorphism $f : {\mathfrak L} \to {\mathfrak L}$, and a bijection $\sigma: I \to J$ such that $f({\mathfrak L}_i)={\mathfrak L}'_{\sigma(i)}$ for all $i \in I$.
\end{definition}


\begin{theorem}
Let ${\mathfrak L}$ be a Leibniz algebra and consider two multiplicative bases  $\B:=\{v_i\}_{i \in I}$ and $\B':=\{v_j'\}_{j \in J}$. Consider the following assertions:
\begin{itemize}
\item[i.] The bases $\B$ and $\B'$ are equivalent.
\item[ii.] The graphs $\Gamma({\mathfrak L},\B)$ and $\Gamma({\mathfrak L},\B')$ are isomorphic.
\item[iii.] The decompositions $$\Upsilon := {\mathfrak L} = \bigoplus\limits_{[v_i] \in V/\sim} {\mathfrak L}_{C_{[v_i]}} \hspace{0.6cm} \text{and} \hspace{0.6cm} \Upsilon' :={\mathfrak L}= \bigoplus\limits_{[v'_j] \in V'/\sim} {\mathfrak L}_{C'_{[v'_j]}},$$  corresponding to $\B$ and $\B'$ are equivalent.
\end{itemize}
Then i. implies ii. and  iii.   
\end{theorem}

\begin{proof}
The implication from  i.\ to ii.\ was proved in Lemma \ref{isomorphism}.  Suppose that $f : {\mathfrak L} \to {\mathfrak L}$ is an automorphism satisfying $f(\B)=\B'$. By the implication from  i.\ to ii., we know that 
$\Gamma({\mathfrak L},\B)$ and $\Gamma({\mathfrak L},\B')$ are isomorphic via $f$ and thus $f([v])=[f(v)]$, for all $v\in V=\B$. It follows that $f\left({\mathfrak L}_{C_{[v]}}\right)={\mathfrak L}_{C_{[f(v)]}}$, for all $[v]\in V/\sim$, which proves that the decompositions of ${\mathfrak L}$ corresponding to $\B$ and $\B'$ are equivalent.

\end{proof}

\section{Minimality and Weak Symmetry}
Consider a graph $(V,E)$. Given two vertices $v_i,v_j \in V$, a {\em directed path} from $v_i$ to $v_j$ is a sequence of vertices $(v_{i_1},\dots,v_{i_n})$ with $v_{i_1} =v_i$, $v_{i_n} =v_j$ and  such that $(v_{i_{r}},v_{i_{r+1}})\in E$, for every $1\le r \le n-1$. We also recall that a graph $(V,E)$ is {\em symmetric} if $(v_i,v_j) \in E$ for every $(v_j,v_i)\in E$. Then, a weaker concept can be introduced as follows:

\begin{definition}\label{Def_weakly_symmetric}\rm
We say that the graph $(V,E)$ is {\em weakly symmetric} if for every $(v_j,v_i)\in E$ there exists a directed path from $v_i$ to $v_j$.
\end{definition}

\noindent In particular, we have that every symmetric graph is weakly symmetric.

\begin{example}
The following graphs are  weakly symmetric:
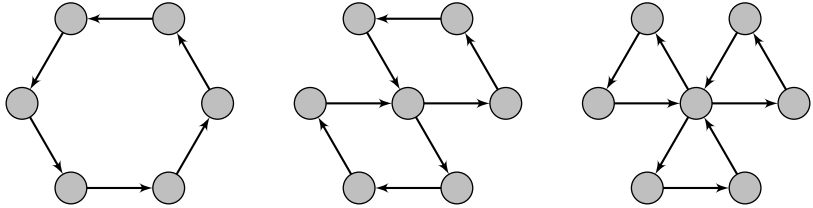
\begin{center}
\begin{tikzpicture}
	\SetGraphUnit{1.3}
	\SetVertexNormal[FillColor = gray!50,
					MinSize    = 12pt]
	\SetVertexNoLabel
	\Vertices{circle}{A,B,C,D,E,F}
	\tikzset{EdgeStyle/.style={->,> = latex'}}
	\Edges(A,B,C,D,E,F,A)
\end{tikzpicture}
\qquad
\begin{tikzpicture}
	\SetGraphUnit{1.3}
	\SetVertexNormal[FillColor = gray!50,
					MinSize    = 12pt]
	\SetVertexNoLabel
	\Vertices{circle}{A,B,C,D,E,F}
	\Vertex[x=0,y=0]{G}
	\tikzset{EdgeStyle/.style={->,> = latex'}}
	\Edges(A,B,C,G,F,E,D,G,A)
\end{tikzpicture}
\qquad
\begin{tikzpicture}
	\SetGraphUnit{1.3}
	\SetVertexNormal[FillColor = gray!50,
					MinSize    = 12pt]
	\SetVertexNoLabel
	\Vertices{circle}{A,B,C,D,E,F}
	\Vertex[x=0,y=0]{G}
	\tikzset{EdgeStyle/.style={->,> = latex'}}
	\Edges(A,B,G,C,D,G,E,F,G,A)
\end{tikzpicture}
\end{center}
\end{example}

\begin{definition}\label{def_strongly_connected}\rm
Given a  graph $(V,E)$, we say that two vertices $v_i,v_j\in V$ are {\em strongly connected} if there exists a directed path from $v_i$ to $v_j$ and one from $v_j$ to $v_i$. (As usual, we consider that there is a directed path using no edges from any vertex to itself.)
Additionally, we say that the graph $(V,E)$ is {\em strongly connected} if any two vertices of $V$ are strongly connected.
\end{definition}

\begin{remark}\label{Remark41} \rm
Note that a graph is strongly connected if and only if it is connected and weakly symmetric. Indeed, if any pair of vertices of a graph is strongly connected, then clearly the graph is connected and weakly symmetric. Conversely, suppose that $(V, E)$ is connected and weakly symmetric. Given vertices $v_i,v_j\in V$, let  $(v_i=v_{i_1},\dots,v_{i_n}=v_j)$ be an undirected path from $v_i$ to $v_j$. For each $1\le r \le n-1$, if $(v_{i_{r}},v_{i_{r+1}})\not\in E$, then we must have $(v_{i_{r+1}},v_{i_{r}})\in E$ and by weak symmetry there is a directed path $(v_{i_{r}}=v_{k_1},\dots,v_{k_m}=v_{i_{r+1}})$ from $v_{i_{r}}$ to $v_{i_{r+1}}$; thus we can replace in the given undirected path from $v_i$ to $v_j$ the edge $(v_{i_{r+1}},v_{i_{r}})$ with the directed path $(v_{i_{r}}=v_{k_1},\dots,v_{k_m}=v_{i_{r+1}})$. Proceeding in this fashion for all $1\le r \le n-1$, we end up with a directed path from $v_i$ to $v_j$, as required.
\end{remark}

In the case of Leibniz algebras, there can be no edges of type $(e_i,u_j)$, with $e_i \in \B_{\mathfrak I}$,  $u_j \in \B_{\mathfrak V}$. So the notion of weak symmetry is not well adjusted to the graph associated to a Leibniz algebra ${\mathfrak L}$. Due to this fact, we consider two subgraphs, one related to the ideal ${\mathfrak I}$ and the other to  ${\mathfrak V}$, isomorphic to the Lie algebra ${\mathfrak L}/{\mathfrak I}$.


\begin{definition}\label{def_subgraph_associated}\rm
Let ${\mathfrak L}$ be a Leibniz algebra with a multiplicative basis $\B$.  By a \emph{subgraph} of the  directed graph $\Gamma({\mathfrak L},\B) := (V,E)$ we mean a vertex-induced subgraph $(V',E')$, i.e.\ a 
pair $(V',E')$ with $V'\subset V$ and $E'\subset E$ such that the edge $(u,v)$ of $E$ belongs to $E'$ if and only if $u, v \in V'$.
\end{definition}

\noindent 
Let ${\mathfrak L}= {\mathfrak I} \oplus {\mathfrak V}$ be a Leibniz algebra with multiplicative basis ${\mathcal B}={\mathcal B}_{\mathfrak I} \; \dot{\cup} \; {\mathcal B}_{\mathfrak V}$, where $\B_{\mathfrak I} =\{e_i\}_{i \in I}$ and  $\B_{\mathfrak V}=\{u_j\}_{j \in J}$. In the following, we will consider the two subgraphs  $\Gamma({\mathfrak L},\B_{\mathfrak I})$ and  $\Gamma({\mathfrak L},\B_{\mathfrak V})$ of the associated graph $\Gamma({\mathfrak L},\B)$. These can be obtained from $\Gamma({\mathfrak L},\B)$ by removing the edges $(u_j,e_i)$ with $u_j \in \B_{\mathfrak V}$ and $e_i \in \B_{\mathfrak I}$. The set of vertices of the subgraph $\Gamma({\mathfrak L},\B_{\mathfrak I})$ is $\B_{\mathfrak I}$ while for the subgraph $\Gamma({\mathfrak L},\B_{\mathfrak V})$ this set is $\B_{\mathfrak V}$.

\begin{remark} \rm More precisely, 
taking into account Definition \ref{def_mult_Leib} and Remark \ref{Remark21} we observe that the subgraph $\Gamma({\mathfrak L},\B_{\mathfrak I}) $ is $ (V',E')$ where $V'=\B_{\mathfrak I}$ and $$E' := \{(e_i,e_k) \in V \times V : [e_i,u_j]\in \mathbb{F}^{\times}e_k, \hspace{0.1cm} \mbox{\rm for some } u_j \in {\mathfrak V}\}$$ and   the subgraph $\Gamma({\mathfrak L},\B_{\mathfrak V}) $ is $ (V'' ,E'')$ where $V''=\B_{\mathfrak V}$ and $$E'' := \{(u_j,u_l) \in V \times V : \{[u_j,u_k],[u_k,u_j]\} \cap \mathbb{F}^{\times}u_l \neq \emptyset \hspace{0.1cm} \mbox{\rm for some } u_k \in {\mathfrak V}\}.$$
\end{remark}

\begin{example}
For the complex Leibniz algebra ${\mathfrak L}$ of the Example \ref{Ejemplo_grafo} we get
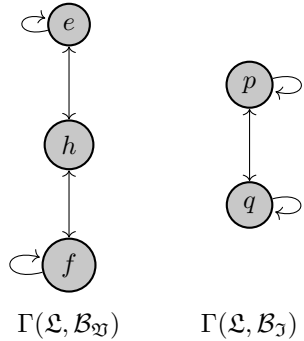
\begin{center}
\begin{tikzpicture}[scale=0.8, transparency group=knockout]
\node (T) at (-1,-5) {$\Gamma({\mathfrak L},\B_{\mathfrak V})$};
\node (T) at (2,-5) {$\Gamma({\mathfrak L},\B_{\mathfrak I})$};
\begin{scope}[every node/.style={circle,thick,draw}]
\node[shape=circle,fill={rgb,255:gray,300; white,50},draw=black] (E) at (-1,0) {$e$};
\node[shape=circle,fill={rgb,255:gray,300; white,50},draw=black] (P) at (2,-1) {$p$};
\node[shape=circle,fill={rgb,255:gray,300; white,50},draw=black] (H) at (-1,-2) {$h$};
\node[shape=circle,fill={rgb,255:gray,300; white,50},draw=black] (F) at (-1,-4) {$f$};
\node[shape=circle,fill={rgb,255:gray,300; white,50},draw=black] (Q) at (2,-3) {$q$};

\draw[<->] (E) edge (H);
\draw[<->] (H) edge (F);
\draw[<->] (P) edge (Q);

\path (E) edge [loop left] (E);
\path (P) edge [loop right] (P);
\path (Q) edge [loop right] (Q);
\path (F) edge [loop left] (F);
\end{scope}
\end{tikzpicture}
\end{center}
\end{example}

In what follows, we denote by $\I(v)$ the ideal of a Leibniz algebra ${\mathfrak L}$ generated by $v\in {\mathfrak L}$.

\begin{lemma}\label{LEMMA 4.2}
Let ${\mathfrak L}$ be a Leibniz algebra with a multiplicative basis $\B:=\{v_i\}_{i \in I}$.
Given $v_i,v_j \in \B_K$, we have that $v_j \in \I(v_i)$ if and only if there exists a directed path from $v_i$ to $v_j$ in $\Gamma({\mathfrak L},\B_K)$, for whatever $K \in \{{\mathfrak I},{\mathfrak V}\}$.
\end{lemma}

\begin{proof}
We consider the direct implication first. For $x, y\in{\mathfrak L}$, write $f(x,y)$ to denote either one of the elements $[x,y]$ or $[y,x]$. Then 
\begin{align}\label{Ivi}
\I(v_i)=\sum \IF f( f(\cdots f(f(v_i,w_1),w_2),\cdots),w_k),
\end{align}
where the sum is over all $k\geq 0$ (in case $k=0$, by convention, the corresponding subspace is $\IF v_i$), all $w_1,\ldots, w_k\in\B$ and all possible choices for $f(x,y)=[x,y]$ or $f(x,y)=[y,x]$. Indeed, the right-hand side of \eqref{Ivi} is a subspace of ${\mathfrak L}$ which is closed under left and right products with elements of ${\mathfrak L}$, thus it is an ideal. Since it contains $v_i$, it contains $\I(v_i)$ and in fact it must equal $\I(v_i)$ as every summand from the right-hand side of \eqref{Ivi} is contained in $\I(v_i)$. Observe also that, since $\B$ is multiplicative and $v_i, w_1,\ldots, w_k\in\B$, each term on the right-hand side of \eqref{Ivi} is either $0$ or a nonzero multiple of $v_l$, for some $v_l\in\B$. It follows that if $v_i,v_j\in\B$ and $v_j \in \I(v_i)$, then there exist $k\geq 0$, $w_1,\dots,w_k \in \B$ and $\lambda\in\IF^\times$ such that
\begin{align*}
f( f(\cdots f(f(v_i,w_1),w_2),\cdots),w_k)=\lambda v_j,
\end{align*} 
where above each occurrence of $f(x,y)$ has been assigned a specific interpretation as $[x,y]$ or $[y,x]$, which we fix. For each $1\leq l\leq k$, let $v_{i_l}$ be the unique element of $\B$ such that $f( f(\cdots f(f(v_i,w_1),w_2),\cdots),w_l)$ is a nonzero multiple of $v_{i_l}$. Then there is a directed path $(v_i,v_{i_1}, \ldots, v_{i_{k-1}}, v_j)$ from $v_i$ to $v_j$ in $\Gamma({\mathfrak L},\B)$. To finish, we just need to argue that this path is contained in the subgraph $\Gamma({\mathfrak L},\B_K)$, assuming that $v_i,v_j \in \B_K$. This is clear as $\B=\B_{\mathfrak I}\,\dot{\cup}\,\B_{\mathfrak V}$ and there are no edges in $\Gamma({\mathfrak L},\B)$ from a vertex in $\B_{\mathfrak I}$ to a vertex in $\B_{\mathfrak V}$.

For the converse implication, assume that $(v_i=v_{i_1}, \ldots, v_{i_{k}}=v_j)$ is a directed path from $v_i$ to $v_j$. Since $(v_{i_1},v_{i_2})$ is an edge, there exist $\lambda\in\IF^\times$ and $w_1\in\B$ such that $f(v_{i_1}, w_1)=\lambda v_{i_2}$, so $v_{i_2}=f(v_{i_1}, \lambda^{-1}w_1)\in \I(v_i)$. Proceeding successively as above along the path, we conclude that $v_{i_l}\in \I(v_i)$ for all $1\leq l\leq k$, so $v_j \in \I(v_i)$.

\end{proof}

\begin{definition}\label{def_base_debil_division}\rm
Let ${\mathfrak L} = {\mathfrak I} \oplus {\mathfrak V}$ be a Leibniz algebra with a multiplicative basis $\B={\mathcal B}_{\mathfrak I} \; \dot{\cup} \; {\mathcal B}_{\mathfrak V}$, where ${\mathcal B}_{\mathfrak I}=\{e_k\}_{k \in K}$ and ${\mathcal B}_{\mathfrak V} = \{u_j\}_{j \in J}$. It is said that $\B$ is of {\it weak division} if
\begin{itemize}
\item $0 \neq [e_i,u_j] = \lambda e_k$ implies $e_i \in \mathcal{I}(e_k)$.
\item $0 \neq [u_i,u_k] = \lambda u_p$ implies $u_i,u_k \in \mathcal{I}(u_p)$.
\end{itemize}
\end{definition}

\begin{proposition}\label{Prop 4.1}
Let ${\mathfrak L}$ be a Leibniz algebra admitting a multiplicative basis $\B={\mathcal B}_{\mathfrak I} \; \dot{\cup} \; {\mathcal B}_{\mathfrak V}$, where ${\mathcal B}_{\mathfrak I}=\{e_k\}_{k \in K}$ and ${\mathcal B}_{\mathfrak V} = \{u_j\}_{j \in J}.$ Under these conditions, $\B$ is of weak division if and only if the subgraph $\Gamma({\mathfrak L},\B_K)$ is weakly symmetric, for $K \in \{{\mathfrak I},{\mathfrak V}\}$.
\end{proposition}
\begin{proof}
Let us suppose that the multiplicative basis $\B$ is of   weak division. First we prove that $\Gamma({\mathfrak L},\B_{\mathfrak V})$ is weakly symmetric. Given an edge $(u_i,u_j)$ of $\Gamma({\mathfrak L},\B_{\mathfrak V})$ we have that there exist $u_k \in \B_{\mathfrak V}$ and $\lambda_j \in \mathbb{F}^{\times}$ such that $[u_i,u_k]=\lambda_j u_j$ or $[u_k,u_i]=\lambda_ju_j$. Since $\B$ is of weak division we get $u_i \in \mathcal{I}(u_j)$, and by Lemma \ref{LEMMA 4.2} there exists a directed path from $u_j$ to $u_i$, so
$\Gamma({\mathfrak L},\B_{\mathfrak V})$ is weakly symmetric.

Now we prove that $\Gamma({\mathfrak L},\B_{\mathfrak I})$ is weakly symmetric. Let us take an edge $(e_i,e_j)$ of $\Gamma({\mathfrak L},\B_{\mathfrak I})$. So there exists $u_k \in \B_{\mathfrak V}$ such that $[e_i,u_k]=\lambda_je_j$, for some $\lambda_j \in \mathbb{F}^{\times}$. Since $\B$ is of weak division we get $e_i \in \mathcal{I}(e_j)$. As in the previous case, by Lemma \ref{LEMMA 4.2} we conclude $\Gamma({\mathfrak L},\B_{\mathfrak I})$ is weakly symmetric.

Conversely, let us suppose that the subgraph $\Gamma({\mathfrak L},\B_K)$ is weakly symmetric, for $K \in \{{\mathfrak I},{\mathfrak V}\}$. We want to show that $\B$ is of weak division.
Given $e_i,e_k \in \B_{\mathfrak I}, u_j \in \B_{\mathfrak V}$ such that $0 \neq [e_i,u_j] = \lambda e_k$. Then $(e_i,e_k)$ is an edge of $\Gamma({\mathfrak L},\B_{\mathfrak I})$. Since $\Gamma({\mathfrak L},\B_{\mathfrak I})$ is weakly symmetric, there exists a directed path from $e_k$ to $e_i$. Then, by Lemma \ref{LEMMA 4.2} it follows that $e_i \in \I(e_k)$. In a similar way, for $u_i, u_k, u_p \in \B_{\mathfrak V}$ such that $0 \neq [u_i,u_k] = \lambda u_p$, we get that $(u_i,u_p)$ and $(u_k,u_p)$ are edges of $\Gamma({\mathfrak L},\B_{\mathfrak V})$. Analogously to the previous case, Lemma \ref{LEMMA 4.2} lets us assert that $u_i, u_k \in \I(u_p)$. In conclusion, $\B$ is of weak division, as required.
\end{proof}


\begin{proposition}\label{Proposicion 6}
Let ${\mathfrak L}$ be a Leibniz algebra admitting a multiplicative basis $\B$. If $\Gamma({\mathfrak L},\B_{\mathfrak I})$ is strongly connected, then any ideal $I$ of ${\mathfrak L}$ such that $I\cap\B_{\mathfrak I}\neq\emptyset$ verifies ${\mathfrak I}\subset I$.
\end{proposition}
\begin{proof}
Let $e_{i_0} \in \B_{\mathfrak I}\cap I$. 
By hypothesis, there exists a directed path from $e_{i_0}$ to $e_i$ in $\Gamma({\mathfrak L},\B_{\mathfrak I})$, for all $e_i\in\B_{\mathfrak I}$. Applying Lemma \ref{LEMMA 4.2} we have $e_i \in {\mathcal I}(e_{i_0}) \subset I$, therefore
${\mathfrak I} \subset I.$
\end{proof}

\begin{proposition}\label{Proposicion 5}
Let ${\mathfrak L}$ be a Leibniz algebra admitting a multiplicative basis $\B$. If the subgraph $\Gamma({\mathfrak L},\B_{\mathfrak V})$ 
is strongly connected, then any ideal $I$ of ${\mathfrak L}$ such that $I \cap \B_{\mathfrak V}\neq \emptyset$ verifies $I={\mathfrak L}$.
\end{proposition}
\begin{proof}
We will prove that ${\mathfrak L} = {\mathfrak I} \oplus {\mathfrak V} \subset I$. Let $u_{j_0}\in \B_{\mathfrak V}\cap I$.
By hypothesis, there exists a directed path from $u_{j_0}$ to $u_j$ in $\Gamma({\mathfrak L},\B_{\mathfrak V})$, for all $u_j \in \B_{\mathfrak V}$. Therefore, by Lemma~\ref{LEMMA 4.2} we have $u_j \in {\mathcal I}(u_{j_0}) \subset I$, so consequently
\begin{equation}\label{EQQ1}
{\mathfrak V} \subset I.
\end{equation}
Finally, since ${\mathfrak I} \subset [{\mathfrak I},{\mathfrak V}]+[{\mathfrak V},{\mathfrak V}]$, by Equation \eqref{EQQ1} we get ${\mathfrak I} \subset I$, which proves $I={\mathfrak L}$, as desired.
\end{proof}

%
%
%
%
%

\begin{definition}\rm
A Leibniz algebra ${\mathfrak L}$ admitting a multiplicative basis $\B$ is {\em minimal} if its only non-zero ideals admitting a multiplicative basis contained in $\B$ are ${\mathfrak I}$ and ${\mathfrak L}$.
\end{definition}

\begin{theorem}\label{theorem}
Let ${\mathfrak L}$ be a Leibniz algebra with multiplicative basis $\B={\mathcal B}_{\mathfrak I} \; \dot{\cup} \; {\mathcal B}_{\mathfrak V}$. Then the following statements are equivalent:
\begin{itemize}
\item[i.] ${\mathfrak L}$ is minimal.
\item[ii.] $\Gamma({\mathfrak L},\B_K)$ is strongly connected, for $K \in \{{\mathfrak I},{\mathfrak V}\}$.
\item[iii.] $\B$ is of weak division and $\Gamma({\mathfrak L},\B_K)$ is connected, for $K \in \{{\mathfrak I},{\mathfrak V}\}$.
\end{itemize}
\end{theorem}

\begin{proof}
The equivalence between ii.\ and iii.\ follows by Proposition \ref{Prop 4.1} and Remark~\ref{Remark41}.
Now, we prove that ii.\ implies i. Let $I$ be a non-zero ideal of ${\mathfrak L}$ admitting a multiplicative basis contained in $\B$. We have two possibilities, either $I \subset {\mathfrak I}$  or $I \not\subset {\mathfrak I}$. In the first case by Proposition \ref{Proposicion 6} we infer that $I={\mathfrak I}$, and in the second case by Proposition \ref{Proposicion 5} we have $I={\mathfrak L}$; consequently ${\mathfrak L}$ is minimal.

It remains to show that  i. implies ii. Suppose that ${\mathfrak L}$ is minimal. First we take care of the case $K ={\mathfrak I}$. Let $e_i, e_j\in\B_{\mathfrak I}$. The minimality of ${\mathfrak L}$ implies that ${\I}(e_i)={\mathfrak I}$, so in particular $e_j\in\I(e_i)$. Then, by Lemma \ref{LEMMA 4.2}, there exists a directed path from $e_i$ to $e_j$ in $\Gamma({\mathfrak L},\B_{\mathfrak I})$, proving that the latter is strongly connected.
Similarly, given $v_i,v_j\in\B_{\mathfrak V}$, the minimality of ${\mathfrak L}$ implies that ${\I}(v_i)={\mathfrak L}$, so $v_j\in\I(v_i)$. Then, by Lemma \ref{LEMMA 4.2}, there exists a directed path from $v_i$ to $v_j$ in $\Gamma({\mathfrak L},\B_{\mathfrak V})$. This completes the proof.

\end{proof}

\begin{example}
For the $n$-dimensional Leibniz algebra ${\mathfrak L}$ from Example \ref{EJEMM_24} with basis $\mathcal{B} := \{e,f,h,x_0,x_1,\dots, x_{n-4}\}$ the subgraphs $\Gamma({\mathfrak L},\mathcal{B}_{{\mathfrak I}})$ and $\Gamma({\mathfrak L},\mathcal{B}_{{\mathfrak V}})$ are:

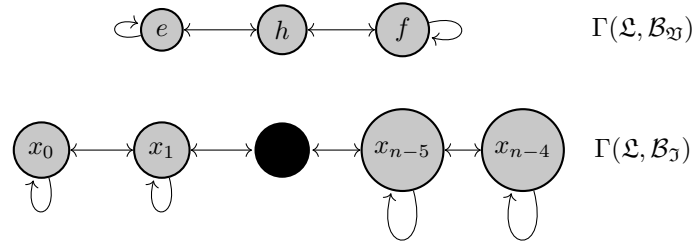
\begin{center}
\begin{tikzpicture}[scale=0.8, transparency group=knockout]
\node (T) at (7,2) {$\Gamma({\mathfrak L},\B_{\mathfrak V})$};
\node (T) at (7,0) {$\Gamma({\mathfrak L},\B_{\mathfrak I})$};
\begin{scope}[every node/.style={circle,thick,draw}]
\node[shape=circle,fill={rgb,255:gray,300; white,50},draw=black] (E) at (-1,2) {$e$};
\node[shape=circle,fill={rgb,255:gray,300; white,50},draw=black] (H) at (1,2) {$h$};
\node[shape=circle,fill={rgb,255:gray,300; white,50},draw=black] (F) at (3,2) {$f$};
\node[shape=circle,fill={rgb,255:gray,300; white,50},draw=black] (X0) at (-3,0) {$x_0$};
\node[shape=circle,fill={rgb,255:gray,300; white,50},draw=black] (X1) at (-1,0) {$x_1$};
\node[shape=circle,fill={rgb,255:gray,300; white,50},draw=black] (XN-5) at (3,0) {$x_{n-5}$};
\node[shape=circle,fill={rgb,255:gray,300; white,50},draw=black] (XN-4) at (5,0) {$x_{n-4}$};

\draw[<->] (E) edge (H);
\draw[<->] (H) edge (F);
\draw[<->] (X0) edge (X1);
\draw[<->] (X1) edge (0.5,0);
\draw[<->] (XN-5) edge (1.5,0);
\draw[<->] (XN-5) edge (XN-4);


\node[fill,opacity=0,text opacity=1] at (1,0) {\ldots};

\path (E) edge [loop left] (E);
\path (F) edge [loop right] (F);
\path (X0) edge [loop below] (X0);
\path (X1) edge [loop below] (X1);
\path (XN-5) edge [loop below] (XN-5);
\path (XN-4) edge [loop below] (XN-4);
\end{scope}
\end{tikzpicture}
\end{center}

\noindent Since the subgraphs $\Gamma({\mathfrak L},\B_{\mathfrak I})$ and $\Gamma({\mathfrak L},\B_{\mathfrak V})$ are strongly connected, Theorem \ref{theorem} shows that $ {\mathfrak L} $ is minimal.
\end{example}

\bigskip

{\bf Acknowledgment.} The authors would like to thank the referee for the detailed reading of this work and for the suggestions
which have improved the final version of the paper.


\bigskip


\begin{thebibliography}{99}


\bibitem{Alb} Albeverio, S.; Omirov, B.A.; Rakhimov, I.S.; Varieties of nilpotent complex Leibniz algebras of dimension less than five. {\em Comm. Algebra} {\bf 33} (2005), 5, 1575-1585.

\bibitem{Ayu99} Ayupov, Sh.A.; Omirov, B.A.; On 3-dimensional Leibniz algebras. {\em Uzbek Math. J.} {\bf 1} (1999), 9-14.

\bibitem{Ayu01} Ayupov, Sh.A.; Omirov, B.A.; On some classes of nilpotent Leibniz algebras. {\em Sib. Math. J. } {\bf 42} (2001), 1, 18-29.

\bibitem{new_book} Ayupov, Sh.A.; Omirov, B.A.; Rakhimov, I.S; Leibniz algebras: Structure and Classification. Taylor \& Francis (2019).

\bibitem{big1} Badawi, A; On the annihilator graph of a commutative ring. {\it Commun. Algebra} {\bf 42} (2014), no. 1, 108--121.


\bibitem{Leibniz3} Bloh. A.; On a generalization of the concept of Lie algebra. {\em Dokl. Akad. Nauk SSSR} {\bf 165} (1965), no. 3, 471–473.



\bibitem{Cabezas} Cabezas, J.M.; Camacho, L.M.; Rodríguez, I.M.; On filiform and 2-filiform Leibniz algebras of maximum length. {\em J. Lie Theory} {\bf 18} (2008), 2, 335-350.

\bibitem{Yo_Leibniz_mult_basis} Calder\'on, A.J.; Leibniz algebras admitting a multiplicative basis. {\em Bull. Malays. Math. Sci. Soc.} {\bf 40} (2017), 679-695.

\bibitem{Yografos} Calder\'on, A.J.; Navarro, F.J.; Bilinear maps and graphs. {\em Discrete Appl. Math.} {\bf 263} (2019), 69--78.

\bibitem{mbasis} Calder\'on, A.J.; Navarro, F.J.; Arbitrary algebras with a multiplicative basis. {\it Linear Algebra Appl.} {\bf 498} (2016), 1, 106-116.


\bibitem{Casas} Casas, J. M.; Corral, N.; On universal central extensions of Leibniz algebras. {\em Comm. Algebra} {\bf 37} (2009), 6, 2104-2120.

\bibitem{Wedderburn} Cohn, P.M.; Basic Algebra: Groups, Rings and Fields. Springer, London (2003)


\bibitem{big2} Chih, T.;  Plessas, D; Graphs and their associated inverse semigroups. {\it Discrete Math.} {\bf 340} (2017), no. 10, 2408--2414.

\bibitem{graph3} Das, A.; Subspace inclusion graph of a vector space. {\it Commun. Algebra} {\bf 44} (2016), no. 11, 4724--4731.

\bibitem{graph4} Das, A.; On non-zero component graph of vector spaces over finite fields. {\it J. Algebra Appl.} {\bf 16} (2017), no. 1, DOI: 10.1142/S0219498817500074

\bibitem{f4} Elduque, A.,  Kochetov, M.:  Gradings on simple Lie algebras. Mathematical Surveys and Monographs, 189. American Mathematical Society, Providence, RI; Atlantic Association for Research in the Mathematical Sciences (AARMS), Halifax, NS, 2013. xiv+336 pp.

\bibitem{Jiang} Jiang, Q.F.; Classification of 3-dimensional Leibniz algebras. (Chinese) {\em J. Math. Res. Exposition} {\bf 27} (2007), 4, 677-686.

\bibitem{Liu} Liu, D., Lin, L.; On the toroidal Leibniz algebras. {\em Acta Math. Sin. (Engl. Ser.)} {\bf 24} (2008), 2, 227-240.

\bibitem{Loday} Loday, J.L.; Une version non commutative des alg\`{e}bres de Lie: les alg\`{e}bres de Leibniz. {\em L'Ens. Math.} {\bf 39} (1993), 269-293.

\bibitem{Omirov} Omirov, B.A.: Conjugacy of Cartan subalgebras of complex finite-dimensional Leibniz algebras. {\em J. Algebra} {\bf  302} (2006), 887-896.

\bibitem{ejemploOmirov} Omirov, B.A.; Rakhimov, I.S.; Turdibaev, R.M.; On description of Leibniz algebras corresponding to $sl_2$. {\em Algebr Represent Theor} {\bf 16} (2013), 1507–1519.

\end{thebibliography}
\end{document}